\documentclass[12pt]{article}
\usepackage{amsmath,amssymb,amsthm}
\usepackage{mathrsfs}
\usepackage{url}
\usepackage{hyperref}
\usepackage[capitalise]{cleveref}
\usepackage[margin=3cm]{geometry}
\usepackage{tikz}
\usetikzlibrary{positioning,arrows.meta}
\usetikzlibrary{positioning,shapes.geometric}
\usepackage{verbatim}
\usepackage{enumerate}
\usepackage{subcaption}

\newtheorem{theorem}{Theorem}[section]
\newtheorem{lemma}[theorem]{Lemma}
\newtheorem{corollary}[theorem]{Corollary}

\DeclareMathOperator{\ex}{ex}

\newcommand{\cC}{\mathcal{C}}

\newcommand{\eps}{\varepsilon}

\newcommand{\DT}{\operatorname{DT}}
\newcommand{\Tu}{\operatorname{Tu}}

\title{The typical structure of oriented graphs and digraphs \\
       with forbidden blow-up of transitive tournaments}

\author{Meili Liang, Yue Guan, Ruiling Zheng, Jianxi Liu\thanks{School of Mathematics and Statistics, Guangdong University of Foreign Studies, Guangzhou, China.
Email: jxliu@gdufs.edu.cn. Corresponding author.}}

\begin{document}

\maketitle

\begin{abstract}
For integers $r\ge 2$, $t\ge 1$ and a real number $a\in(3/2,2]$, we study the typical structure of oriented graphs and digraphs that do not contain a blow-up $T_{r+1}^t$ of a transitive tournament. 
We prove that almost every $T_{r+1}^t$-free oriented graph on $n$ vertices admits an $r$-partition $V_1\cup\cdots\cup V_r$ such that each induced subgraph $G[V_i]$ is $T_2^t$-free, and the same holds for almost every $T_{r+1}^t$-free digraph. 
Consequently, the number $f(n,T_{r+1}^t)$ of labelled $T_{r+1}^t$-free oriented graphs satisfies $f(n,T_{r+1}^t)=|\mathcal{P}_{n,r,t}|(1+o(1))$, where $\mathcal{P}_{n,r,t}$ is the family of oriented graphs admitting such an $r$-partition with each part $T_2^t$-free; an analogous statement holds for digraphs. 
When $t=1$ this recovers the result of K\"uhn, Osthus, Townsend and Zhao (2017) that almost all $T_{r+1}$-free oriented graphs (resp. digraphs) are $r$-partite, thereby confirming a generalised form of Cherlin's conjecture.
Our proof combines the hypergraph container method, a weighted Erd\H{o}s–Stone theorem, and a stability analysis for near-extremal $T_{r+1}^t$-free digraphs.
\end{abstract}



\section{Introduction}

Given a fixed (di)graph $H$, a (di)graph is called \emph{$H$-free} if it contains no copy of $H$. Extremal graph theory addresses two fundamental questions: what is the maximum number of edges in an $H$-free graph on $n$ vertices, and what does a typical $H$-free graph look like? For undirected graphs, Erdős, Kleitman and Rothschild \cite{erdos1976asymptotic} initiated the study of the second question by proving that almost all triangle‑free graphs are bipartite, and Kolaitis, Prömel and Rothschild \cite{kolaitis1987k} extended this to $K_{r+1}$-free graphs, showing that almost all are $r$-partite.

For directed graphs (digraphs) and oriented graphs, the situation is far more complex. A digraph is a set of vertices together with a set of ordered pairs of distinct vertices; an oriented graph is a digraph with at most one arc in each direction. The transitive tournament $T_k$ is the orientation of a complete graph that is transitive (see Figure 1 for examples). In his classification of countable homogeneous oriented graphs, Cherlin \cite{cherlin1998classification} made several conjectures, among them that almost all $T_3$-free oriented graphs are tripartite. This was proved by Kühn, Osthus, Townsend and Zhao \cite{kuhn2017structure} (KOTZ), who also showed that for any $k\ge 2$, almost all $T_{k+1}$-free oriented graphs and digraphs are $k$-partite.

A natural generalisation is to forbid a \emph{blow-up} of a transitive tournament. For integers $r,t\ge 1$, let $T_{r+1}^t$ denote the digraph obtained from $T_{r+1}$ by replacing each vertex with an independent set of size $t$ and adding all possible arcs from the $i$-th part to the $j$-th part whenever $ij$ is an arc of $T_{r+1}$  (see Figure 1 for examples). Thus $T_{r+1}^t$ is an oriented graph (no $2$-cycles). Clearly $T_{r+1}^t$ contains $T_{r+1}$ (take one vertex from each part), so any $T_{r+1}$-free digraph is also $T_{r+1}^t$-free, but the converse is false.

\begin{figure}[ht]
\centering

% 第一行：T2, T3, T4
\begin{subfigure}{0.3\textwidth}
\centering
\begin{tikzpicture}[
    vertex/.style={circle, draw, fill=black, minimum size=5pt, inner sep=0pt},
    arc/.style={-{Stealth[length=3pt, width=2pt]}, thick}
]
\node[vertex, label=above:$1$] (v1) at (0,0) {};
\node[vertex, label=above:$2$] (v2) at (1.5,0) {};
\draw[arc] (v1) -- (v2);
\end{tikzpicture}
\caption{$T_2$}
\label{fig:T2}
\end{subfigure}
\hfill
\begin{subfigure}{0.3\textwidth}
\centering
\begin{tikzpicture}[
    vertex/.style={circle, draw, fill=black, minimum size=5pt, inner sep=0pt},
    arc/.style={-{Stealth[length=3pt, width=2pt]}, thick}
]
\node[vertex, label=above:$1$] (v1) at (0,0) {};
\node[vertex, label=above:$2$] (v2) at (1.5,1.5) {};
\node[vertex, label=above:$3$] (v3) at (3,0) {};
\draw[arc] (v1) -- (v2);
\draw[arc] (v1) -- (v3);
\draw[arc] (v2) -- (v3);
\end{tikzpicture}
\caption{$T_3$}
\label{fig:T3}
\end{subfigure}
\hfill
\begin{subfigure}{0.3\textwidth}
\centering
\begin{tikzpicture}[
    vertex/.style={circle, draw, fill=black, minimum size=5pt, inner sep=0pt},
    arc/.style={-{Stealth[length=3pt, width=2pt]}, thick}
]
\node[vertex, label=above:$1$] (u1) at (0,0) {};
\node[vertex, label=above:$2$] (u2) at (1.5,1.2) {};
\node[vertex, label=above:$3$] (u3) at (3,0) {};
\node[vertex, label=right:$4$] (u4) at (1.5,-1.2) {};
\draw[arc] (u1) -- (u2);
\draw[arc] (u1) -- (u3);
\draw[arc] (u1) -- (u4);
\draw[arc] (u2) -- (u3);
\draw[arc] (u2) -- (u4);
\draw[arc] (u3) -- (u4);
\end{tikzpicture}
\caption{$T_4$}
\label{fig:T4}
\end{subfigure}

\vspace{0.5cm}

% 第二行：T2^2, T3^2, T4^2
\begin{subfigure}{0.3\textwidth}
\centering
\begin{tikzpicture}[
    vertex/.style={circle, draw, fill=black, minimum size=4pt, inner sep=0pt},
    cluster/.style={ellipse, draw, dashed, fill=gray!15, inner sep=2pt},
    arc/.style={-{Stealth[length=3pt, width=2pt]}, thick}
]
% T2^2: two clusters, left and right, each with two vertices (vertical pairs)
\draw[cluster] (0,0) ellipse (0.5 and 0.9);
\node[vertex] (a1) at (0,-0.35) {};
\node[vertex] (a2) at (0,0.35) {};
\node at (0,-1.0) {$V_1$};

\draw[cluster] (2.5,0) ellipse (0.5 and 0.9);
\node[vertex] (b1) at (2.5,-0.35) {};
\node[vertex] (b2) at (2.5,0.35) {};
\node at (2.5,-1.0) {$V_2$};
% All arcs from V1 to V2
\foreach \x in {a1,a2} {
  \foreach \y in {b1,b2} {
    \draw[arc] (\x) -- (\y);
}
}
\end{tikzpicture}
\caption{$T_2^2$ (blow-up with $t=2$)}
\label{fig:T2sq}
\end{subfigure}
\hfill
\begin{subfigure}{0.3\textwidth}
\centering
\begin{tikzpicture}[
    vertex/.style={circle, draw, fill=black, minimum size=4pt, inner sep=0pt},
    cluster/.style={ellipse, draw, dashed, fill=gray!15, inner sep=2pt},
    arc/.style={-{Stealth[length=3pt, width=2pt]}, thick}
]
% T3^2: positions: left (0,0) vertical, top (1.5,1.5) horizontal, right (3,0) vertical
\draw[cluster] (0,0) ellipse (0.5 and 0.9);
\node[vertex] (a1) at (0,-0.35) {};
\node[vertex] (a2) at (0,0.35) {};
\node at (0,-1.0) {$V_1$};

\draw[cluster] (1.5,1.5) ellipse (0.9 and 0.5);
\node[vertex] (b1) at (1.5-0.35,1.5) {};
\node[vertex] (b2) at (1.5+0.35,1.5) {};
\node at (1.5,2.1) {$V_2$};

\draw[cluster] (3,0) ellipse (0.5 and 0.9);
\node[vertex] (c1) at (3,-0.35) {};
\node[vertex] (c2) at (3,0.35) {};
\node at (3,-1.0) {$V_3$};
\foreach \x in {a1,a2} {
  \foreach \y in {b1,b2} { \draw[arc] (\x) -- (\y);
}
  \foreach \y in {c1,c2} { \draw[arc] (\x) -- (\y);
}
}
\foreach \x in {b1,b2} {
  \foreach \y in {c1,c2} { \draw[arc] (\x) -- (\y);
}
}
\end{tikzpicture}
\caption{$T_3^2$}
\label{fig:T3sq}
\end{subfigure}
\hfill
\begin{subfigure}{0.3\textwidth}
\centering
\begin{tikzpicture}[
    vertex/.style={circle, draw, fill=black, minimum size=4pt, inner sep=0pt},
    cluster/.style={ellipse, draw, dashed, fill=gray!15, inner sep=2pt},
    arc/.style={-{Stealth[length=3pt, width=2pt]}, thick}
]
% T4^2: left vertical, top horizontal, right vertical, bottom horizontal
\draw[cluster] (0,0) ellipse (0.5 and 0.9);
\node[vertex] (a1) at (0,-0.35) {};
\node[vertex] (a2) at (0,0.35) {};
\node at (0,-1.0) {$V_1$};

\draw[cluster] (1.5,1.5) ellipse (0.9 and 0.5);
\node[vertex] (b1) at (1.5-0.35,1.5) {};
\node[vertex] (b2) at (1.5+0.35,1.5) {};
\node at (1.5,2.1) {$V_2$};

\draw[cluster] (3,0) ellipse (0.5 and 0.9);
\node[vertex] (c1) at (3,-0.35) {};
\node[vertex] (c2) at (3,0.35) {};
\node at (3,-1.0) {$V_3$};

\draw[cluster] (1.5,-1.5) ellipse (0.9 and 0.5);
\node[vertex] (d1) at (1.5-0.35,-1.5) {};
\node[vertex] (d2) at (1.5+0.35,-1.5) {};
\node at (1.5,-2.1) {$V_4$};
\foreach \x in {a1,a2} {
  \foreach \y in {b1,b2,c1,c2,d1,d2} { \draw[arc] (\x) -- (\y);
}
}
\foreach \x in {b1,b2} {
  \foreach \y in {c1,c2,d1,d2} { \draw[arc] (\x) -- (\y);
}
}
\foreach \x in {c1,c2} {
  \foreach \y in {d1,d2} { \draw[arc] (\x) -- (\y);
}
}
\end{tikzpicture}
\caption{$T_4^2$}
\label{fig:T4sq}
\end{subfigure}

\caption{Transitive tournaments $T_2$, $T_3$, $T_4$ and their blow-ups $T_2^2$, $T_3^2$, $T_4^2$ (each vertex replaced by an independent set of size $2$, all arcs from earlier parts to later parts).
The dashed ellipses indicate the clusters; vertices are arranged vertically or horizontally for visual clarity.}
\label{fig:blowups}
\end{figure}
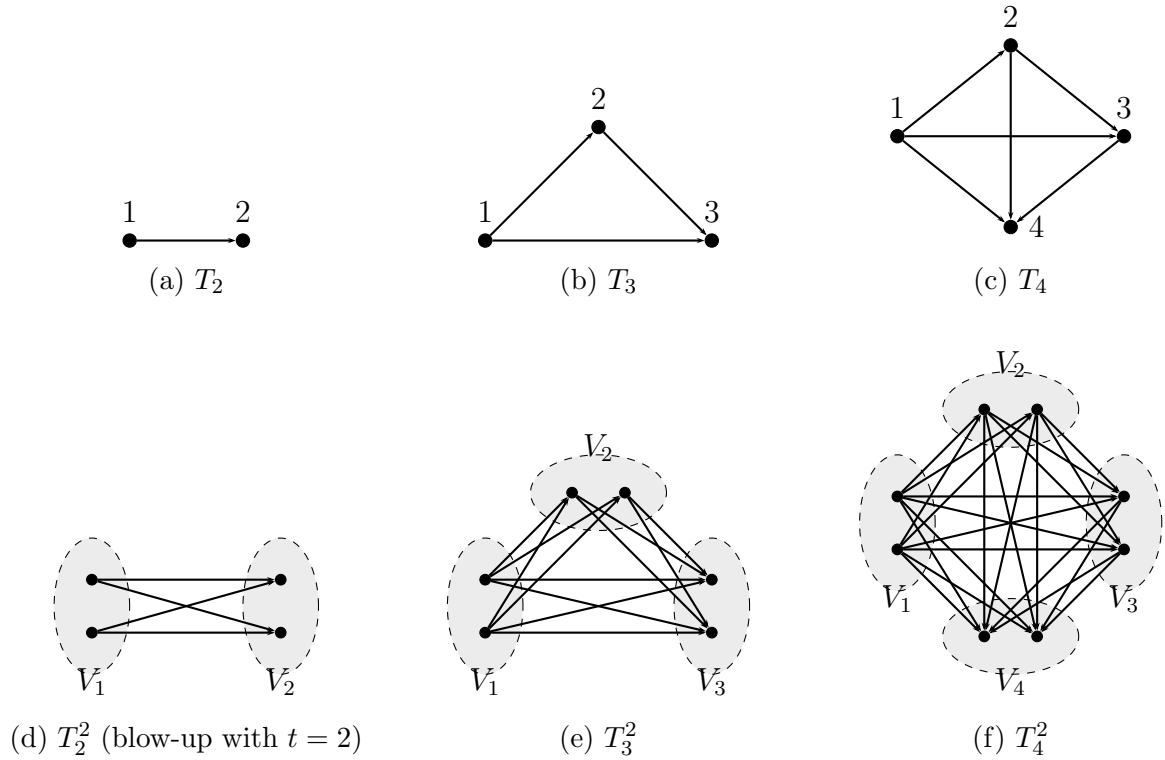

In this paper we determine the typical structure of $T_{r+1}^t$-free oriented graphs and digraphs. Perhaps surprisingly, the presence of the blow-up does not force the graph to be strictly $r$-partite when $t\ge 2$; rather, it allows a limited number of internal edges, provided those edges do not themselves contain a copy of $T_2^t$. Our main result is the following.

\begin{theorem}[Main Theorem]\label{thm:main}
Let $r\ge 2$, $t\ge 1$ be integers and $a\in(3/2,2]$. 
\begin{enumerate}
\item[(i)] Almost every $T_{r+1}^t$-free oriented graph on $n$ vertices admits an $r$-partition\\
$V_1\cup\cdots\cup V_r$ such that each induced subgraph $G[V_i]$ is $T_2^t$-free.
\item[(ii)] Almost every $T_{r+1}^t$-free digraph on $n$ vertices admits an $r$-partition\\
 $V_1\cup\cdots\cup V_r$ such that each induced subgraph $G[V_i]$ is $T_2^t$-free.
\end{enumerate}
Moreover, the number of such graphs satisfies $f(n,T_{r+1}^t) = |\mathcal{P}_{n,r,t}|(1+o(1))$, where $\mathcal{P}_{n,r,t}$ is the family of oriented graphs on $[n]$ that admit an $r$-partition with each part $T_2^t$-free.
\end{theorem}

When $t=1$, $T_2^1$ is simply a single edge, so $T_2^1$-free means “no edges inside a part”, i.e. the part is independent. Thus the theorem recovers the KOTZ result that almost all $T_{r+1}$-free graphs are strictly $r$-partite. For $t\ge 2$, the conclusion is weaker but optimal: there exist many $T_{r+1}^t$-free graphs that are not $r$-partite, and the theorem exactly describes the obstruction.

\section{Preliminaries}
\subsection{Notation}
For a digraph $G=(V,E)$, let $f_1(G)$ be the number of unordered pairs $\{u,v\}$ such that exactly one of $uv$ and $vu$ belongs to $E$, and let $f_2(G)$ be the number of unordered pairs with both $uv$ and $vu$ present.
For a real number $a\ge1$, define the \emph{weighted size}
\[
e_a(G):=a\cdot f_2(G)+f_1(G).
\]
This unifies the treatment of oriented graphs ($a=\log 3$, since each 2‑cycle contributes 3 ways to orient) and digraphs ($a=2$, total number of arcs). Let $\ex_a(n,H)$ denote the maximum $e_a(G)$ over all $H$-free digraphs on $n$ vertices. Let $\Tu_r(n)$ be the $r$-partite Turán graph on $n$ vertices (parts as equal as possible), and let $t_r(n)=e(\Tu_r(n))$. Denote by $\DT_r(n)$ the digraph obtained from $\Tu_r(n)$ by replacing each undirected edge with a pair of opposite arcs. Clearly $\DT_r(n)$ is $r$-partite and $T_{r+1}^t$-free, so $\ex_a(n,T_{r+1}^t)\ge a\,t_r(n)$.

We say that \emph{almost all} graphs in a family $\mathcal{F}$ have property $\mathcal{P}$ if
\[
\lim_{n\to\infty}\frac{|\{G\in\mathcal{F}_n: G\text{ has property }\mathcal{P}\}|}{|\mathcal{F}_n|}=1.
\]

\subsection{Directed regularity and embedding lemmas}
We will need the directed version of Szemerédi's regularity lemma and an embedding lemma for blow-ups. The following formulation is from \cite{alon2004testing}.

\begin{lemma}[Directed regularity lemma]\label{lem:regularity}
For any $\eps\in(0,1)$ and integers $M',M''$, there exist $M$ and $n_0$ such that for any digraph $G$ on $n\ge n_0$ vertices, any initial partition $U_0,U_1,\dots,U_{M''}$, and any $d\in[0,1]$, there exists a partition $V_0,V_1,\dots,V_k$ of $V(G)$ and a spanning subdigraph $G'\subseteq G$ with:
\begin{itemize}
\item $M'\le k\le M$, $|V_0|\le\eps n$, $|V_1|=\dots=|V_k|=\ell$;
\item $G'[V_i]$ is empty for all $i\ge1$;
\item for $1\le i\neq j\le k$, the bipartite digraph $(V_i,V_j)_{G'}$ is either $\eps$-regular with density at least $d$, or has density $0$;
\item every vertex $x$ satisfies $d_{G'}^+(x)>d_G^+(x)-(d+\eps)n$ and similarly for in-degree.
\end{itemize}
\end{lemma}

The \emph{reduced digraph} $R$ has vertex set $\{V_1,\dots,V_k\}$ and an arc $ij$ whenever $(V_i,V_j)_{G'}$ is $\eps$-regular with density $\ge d$.

\begin{lemma}[Embedding lemma]\label{lem:embedding}
For any $d\in(0,1)$ and maximum degree $\Delta\ge1$, there exists $\eps_0>0$ such that the following holds. Let $G$ be a digraph, $R$ the reduced digraph obtained from an $\eps$-regular partition with $\eps\le\eps_0$, cluster size $\ell$, and density parameter $d$. If $H$ is a digraph with $\Delta(H)\le\Delta$ and $H\subseteq R^s$ (the blow-up of $R$ where each vertex is replaced by $s$ independent vertices), and $\ell\ge s/\eps_0$, then $H\subseteq G$.
\hfill (see e.g. \cite[Lemma 4.2]{kuhn2017structure})
\end{lemma}

\subsection{Hypergraph containers for digraphs}
The following container theorem is a key tool. It is a direct consequence of the hypergraph container theorem of Saxton and Thomason \cite{saxton2015hypergraph} and was proved in \cite[Theorem 3.3]{kuhn2017structure} for digraphs. We state it exactly as in \cite{kuhn2017structure}.

\begin{theorem}[K\"uhn, Osthus, Townsend, Zhao, Theorem 3.3]\label{thm:container}
Let $H$ be an oriented graph with $h=v(H)$ and $e(H)\ge 2$, and let $a\ge 1$. For every $\eps>0$ there exists $c>0$ such that for all sufficiently large $N$ there exists a collection $\cC$ of digraphs on vertex set $[N]$ with the following properties.
\begin{enumerate}
\item[(a)] Every $H$-free digraph $I$ on $[N]$ is contained in some $G\in\cC$ (in the labelled sense).
\item[(b)] Every $G\in\cC$ contains at most $\eps N^{h}$ copies of $H$, and $e_a(G)\le \ex_a(N,H)+\eps N^2$.
\item[(c)] $\log|\cC|\le c N^{2-1/m(H)}\log N$, where $m(H)=\max_{H'\subseteq H,\,e(H')>1}\frac{e(H')-1}{v(H')-2}$.
\end{enumerate}
\end{theorem}

For our forbidden digraph $H=T_{r+1}^t$, we need to ensure that the theorem applies. The condition required for the theorem to hold is that $H$ is an oriented graph (no 2‑cycles) – which $T_{r+1}^t$ is – and that the parameter $a$ is chosen appropriately. In our applications we will use $a=\log 3$ for oriented graphs and $a=2$ for digraphs. The theorem does not require any further density condition on $H$; it works for any oriented graph $H$. Therefore we can apply it directly to $T_{r+1}^t$ without additional verification.

\subsection{Removal lemma}
We also need the directed removal lemma of Alon and Shapira \cite{alon2004testing}.

\begin{lemma}[Removal lemma]\label{lem:removal}
For any fixed digraph $H$ on $h$ vertices and any $\gamma>0$, there exists $\eps'>0$ such that for all large $n$, any digraph $G$ on $n$ vertices containing at most $\eps' n^h$ copies of $H$ can be made $H$-free by deleting at most $\gamma n^2$ arcs.
\end{lemma}

\section{Weighted extremal theorem for $T_{r+1}^t$}
In this section we prove a weighted Erd\H{o}s–Stone type theorem for the blow-up $T_{r+1}^t$. The following result extends the extremal result of Brown and Harary \cite{brown1970extremal} and the weighted version proved in \cite{kuhn2017structure}. Note that \cite{kuhn2017structure} actually treats the two discrete cases $a=2$ and $a=\log 3$, but the same proof works for the whole interval $a\in(3/2,2]$ by a continuity argument (see the remark after Lemma 4.1 in \cite{kuhn2017structure}). We therefore state it for the continuous parameter.

\begin{theorem}\label{thm:weighted-extremal}
For any integers $r,t\ge1$, real $a\in(3/2,2]$, and $\gamma>0$, there exists $n_0$ such that for all $n\ge n_0$, every digraph $G$ on $n$ vertices with $e_a(G)\ge a\,t_r(n)+\gamma n^2$ contains $T_{r+1}^t$ as a subdigraph.
\end{theorem}

\begin{proof}
We follow the strategy of the KOTZ proof for $T_{r+1}$. Set $d=\gamma/4$, $\Delta=\Delta(T_{r+1}^t)$, and let $\eps_0$ be given by Lemma~\ref{lem:embedding} for these parameters. Choose $\eps$ small enough so that $\eps\le\eps_0$ and $\delta:=(a-1)d-\eps-a\eps^2/2-a\eps>0$. Let $s=t(r+1)$.

Apply Lemma~\ref{lem:regularity} to $G$ with parameters $\eps,d$ to obtain a partition $V_0,V_1,\dots,V_k$ and a pure digraph $G'\subseteq G$ with clusters of size $\ell$, where $\ell\ge (1-\eps)n/k\ge s/\eps_0$ for large $n$. Let $R$ be the reduced digraph on $k$ vertices.

From the regularity lemma we have
\[
e_a(G)\le e_a(G')+(d+\eps)n^2\le a\eps n^2 + e_a^*(R)\ell^2+(d+\eps)n^2,
\]
where $e_a^*(R)=\sum_{ij\in E(R)}(a d_{ij}^2+d_{ij}^1)$ and $d_{ij}^2,d_{ij}^1$ are the densities of 2‑cycles and 1‑cycles in the pair $(V_i,V_j)_{G'}$. Using $e_a(G)\ge a t_r(n)+\gamma n^2$ and $t_r(n)=\bigl(1-\frac1r\bigr)\frac{n^2}{2}+O(n)$, we obtain after elementary manipulation
\[
e_a^*(R) \ge a\Bigl(1-\frac1r\Bigr)\frac{k^2}{2}+\delta k^2 > a\,t_r(k).
\]

Now $R$ is $T_{r+1}$-free: otherwise, by the embedding lemma, $G'$ would contain $T_{r+1}^t$. The inequality $e_a^*(R)>a\,t_r(k)$ forces $R$ to contain $T_{r+1}$ by the weighted extremal result of \cite{kuhn2017structure} (which is stated for $a=2$ and $a=\log 3$ but extends to the whole interval $(3/2,2]$ by a simple interpolation; see the remark after Lemma 4.1 in that paper). Hence $R\supseteq T_{r+1}$. 

Now we need to lift this to a copy of $T_{r+1}^t$ in $G$. Since $R$ contains $T_{r+1}$, applying an $s$-fold blow-up to $R$ yields that $R^s$ contains $(T_{r+1})^s$. Notice that $H = T_{r+1}^t$ has $v(H) = t(r+1) = s$ vertices. In $(T_{r+1})^s$, each of the $r+1$ parts is an independent set of size $s$. Since $s \ge t$, we can arbitrarily select $t$ vertices from each of these $r+1$ parts to form exactly $T_{r+1}^t$. Thus, we have the rigorous subgraph chain $T_{r+1}^t \subseteq (T_{r+1})^s \subseteq R^s$. Applying the embedding lemma (Lemma~\ref{lem:embedding}) with $H=T_{r+1}^t$ and blow-up parameter $s$ yields $T_{r+1}^t\subseteq G'\subseteq G$. 
\end{proof}

As an immediate corollary we obtain the asymptotics of the weighted Turán number.

\begin{corollary}\label{cor:asymptotic}
For $r,t\ge1$ and $a\in(3/2,2]$, we have $\ex_a(n,T_{r+1}^t)=a\,t_r(n)+o(n^2)$. Moreover, any $n$-vertex $T_{r+1}^t$-free digraph $G$ with $e_a(G)=\ex_a(n,T_{r+1}^t)$ differs from $\DT_r(n)$ by $o(n^2)$ arcs.
\end{corollary}

\section{Stability for $T_{r+1}^t$-free digraphs}\label{sec:stability}

In this section we prove a stability result for $T_{r+1}^t$-free digraphs. It states that any such digraph whose weighted size is close to the extremal value $a\,t_r(n)$ must be close to the complete $r$-partite digraph $\DT_r(n)$.

\begin{theorem}[Stability theorem]\label{thm:stability}
Let $r\ge 2$, $t\ge 1$ and $a\in(3/2,2]$.  For every $\beta>0$ there exist $\gamma>0$ and $n_0$ such that for all $n\ge n_0$, if $G$ is an $n$-vertex $T_{r+1}^t$-free digraph satisfying
\[
e_a(G)\ge a\,t_r(n)-\gamma n^2,
\]
then $G$ can be turned into $\DT_r(n)$ by changing at most $\beta n^2$ arcs.
\end{theorem}

The proof follows a well‑established pattern: we apply the directed regularity lemma to obtain a reduced digraph $R$, transfer the weighted extremal condition to $R$, use a weighted version of the stability theorem for $T_{r+1}$ (which forces $R$ to be close to $\DT_r(m)$), and finally lift this structure back to $G$.  The main difficulty is to control the weights, which is handled by a ``weighted stability lemma'' (Lemma~\ref{lem:weighted-stability}) for $T_{r+1}$‑free digraphs.  We begin with the statement of this auxiliary result.

\subsection{A weighted stability lemma for $T_{r+1}$}\label{subsec:weighted-stability}

For a digraph $R$ on $m$ vertices we denote by $e_a^*(R)$ the quantity
\[
e_a^*(R)=\sum_{ij\in E(R)}w_{ij},
\]
where $w_{ij}\in[1,a]$ is a weight assigned to the arc $ij$. In our application $w_{ij}$ will be the average of $a$ times the density of 2‑cycles plus the density of 1‑cycles over a regular pair. The following lemma is the analogue of the ordinary stability theorem for $T_{r+1}$ but with weights.

\begin{lemma}[Weighted stability lemma]\label{lem:weighted-stability}
For every $r\ge2$, $a\in(3/2,2]$ and $\eta>0$ there exist $\delta>0$ and $m_0$ such that for all $m\ge m_0$ the following holds.  Let $R$ be an $m$-vertex $T_{r+1}$-free digraph and assign to each arc $ij\in E(R)$ a weight $w_{ij}\in[1,a]$ in such a way that
\[
e_a^*(R)\ge a\,t_r(m)-\delta m^2.
\]
Then there exists a partition $U_1,\dots,U_r$ of the vertex set of $R$ with the following properties, up to at most $\eta m^2$ exceptions:
\begin{enumerate}[(i)]
\item if $i,j$ belong to the same class $U_p$, then $R$ contains no arc between $i$ and $j$;
\item if $i\in U_p$, $j\in U_q$ with $p\neq q$, then both arcs $ij$ and $ji$ are present in $R$ and their weights satisfy $w_{ij},w_{ji}\ge a-\eta$.
\end{enumerate}
\end{lemma}

\subsubsection{Proof of Lemma~\ref{lem:weighted-stability}}

We argue by contradiction.  Suppose the statement is false.  Then there exist constants $\eta_0>0$ and a sequence of counterexamples: for every $k$ we can find $m_k\to\infty$, a $T_{r+1}$-free digraph $R_k$ on $m_k$ vertices, and weights $w_{ij}^{(k)}\in[1,a]$ (defined only for arcs $ij\in E(R_k)$; non‑arcs have no weight) such that
\[
e_a^*(R_k)\ge a\,t_r(m_k)-\frac{1}{k}\,m_k^2 \qquad\text{(i.e. }\delta_k=1/k\to0\text{)},
\]
but $R_k$ does **not** admit a partition $U_1,\dots,U_r$ with the required properties for $\eta_0$ (i.e., for every $r$-partition there are at least $\eta_0 m_k^2$ violating pairs). We will derive a contradiction.

\textbf{Step 1: Regularisation of $R_k$.}
Apply the directed regularity lemma (Lemma~\ref{lem:regularity}) to $R_k$ with parameters $\varepsilon$ and $d$ that will be chosen later (very small compared to $\eta_0$). We obtain a partition $V_0^{(k)},V_1^{(k)},\dots,V_{p_k}^{(k)}$ of $V(R_k)$ and a pure digraph $R_k'\subseteq R_k$ with the usual properties: $|V_0^{(k)}|\le\varepsilon m_k$, $|V_1^{(k)}|=\dots=|V_{p_k}^{(k)}|=\ell_k$, $p_k\ge M'$, and for every $i\neq j$, the pair $(V_i^{(k)},V_j^{(k)})_{R_k'}$ is either $\varepsilon$-regular with density at least $d$ or has density $0$. Denote $p=p_k$ and let $\widetilde{R}_k$ be the reduced digraph on $\{1,\dots,p\}$ (arcs correspond to regular pairs of density $\ge d$).

\textbf{Step 2: Transferring weights to $\widetilde{R}_k$.}
For each ordered pair $(i,j)$ that is an arc of $\widetilde{R}_k$, define
\[
\widetilde{w}_{ij}^{(k)} = \frac{1}{|V_i^{(k)}||V_j^{(k)}|}\sum_{u\in V_i^{(k)},v\in V_j^{(k)}} w_{uv}^{(k)},
\]
where $w_{uv}^{(k)}$ is defined only if $uv\in E(R_k)$; for pairs $(u,v)$ that are not arcs we simply do not sum them. Since we only sum over arcs, we have $\widetilde{w}_{ij}^{(k)}\in[1,a]$ because each $w_{uv}^{(k)}\in[1,a]$ and we are averaging over all pairs in the regular pair (including non‑arcs, which are not counted). To handle the contribution of non‑arcs we use the fact that in a regular pair the number of missing arcs is at most $(\varepsilon+d)|V_i||V_j|$, and the weight of each such missing arc would be at most $a$ if we had defined it; however we simply ignore them, and the error can be bounded by a term $O(\varepsilon+d)$ times the total number of pairs. A standard calculation (see e.g. \cite[Lemma 9.2]{saxton2015hypergraph}) gives
\[
e_a^*(\widetilde{R}_k) \ge a\,t_r(p) - \delta_k' p^2,
\]
where $\delta_k'\to0$ as $k\to\infty$ (provided $\varepsilon$ and $d$ are chosen small enough relative to $\delta_k$). In particular, for large $k$ we have $\delta_k'\le \delta_0$ for any prescribed $\delta_0>0$.

\textbf{Step 3: Structure of $\widetilde{R}_k$.}
Because $R_k$ is $T_{r+1}$-free, the reduced digraph $\widetilde{R}_k$ is also $T_{r+1}$-free (otherwise an embedding argument would produce a $T_{r+1}$ in $R_k$). Moreover, the number $p$ of clusters is bounded by some absolute constant $M$ depending only on $\varepsilon$ and the initial parameters; indeed, the regularity lemma guarantees $p\le M$.  Thus $p$ is bounded independently of $k$. For each fixed $p$, there are only finitely many $T_{r+1}$-free digraphs on $p$ vertices. The weighted extremal number $\ex_a(p,T_{r+1})$ equals $a t_r(p)$ and is uniquely attained by $\DT_r(p)$ (by Lemma 4.1 in \cite{kuhn2017structure}, which holds for all $p$). Consequently, if $\delta_k'$ is smaller than the minimum positive difference between $a t_r(p)$ and the weighted size of any non‑extremal $T_{r+1}$-free digraph on $p$ vertices, then $e_a^*(\widetilde{R}_k)\ge a t_r(p)-\delta_k' p^2$ forces $\widetilde{R}_k\cong\DT_r(p)$. Since $\delta_k'\to0$, for sufficiently large $k$ this condition is satisfied, and we obtain $\widetilde{R}_k\cong\DT_r(p)$. Let the parts of this $\DT_r(p)$ be $W_1,\dots,W_r$ (each $W_i$ is a set of cluster indices).

\textbf{Step 4: Lifting to $R_k$.}
Now define a partition $U_1,\dots,U_r$ of $V(R_k)$ by putting every vertex belonging to a cluster with index in $W_i$ into $U_i$, and distributing the vertices of the exceptional set $V_0^{(k)}$ arbitrarily (e.g., equally). We claim that this partition satisfies the required properties with at most $\eta_0 m_k^2$ exceptions, contradicting the choice of $R_k$.

Consider any two vertices $x, y$ not both in $V_0^{(k)}$. If they lie in different clusters $V_i, V_j$ with $i \in W_p, j \in W_q$ and $p \neq q$, then because $\widetilde{R}_k$ has both arcs $ij$ and $ji$, the pair $(V_i, V_j)$ is $\varepsilon$-regular with density at least $d$. Moreover, from $e_a^*(\widetilde{R}_k) = a t_r(p)$ and the uniqueness of the extremal digraph, we actually have $\widetilde{w}_{ij} = a$ and $\widetilde{w}_{ji} = a$ for every cross arc of $\widetilde{R}_k$. We now formalise the distribution of the original weights $w_{uv}$ using Markov's inequality. Define the defect of a pair $(u,v)$ as $X_{uv} = a - w_{uv} \ge 0$. Since $e_a^*(R_k) \ge a\,t_r(m_k) - \delta_k m_k^2$, the total defect over all pairs is bounded by $\sum X_{uv} \le \delta_k m_k^2 + O(\varepsilon m_k^2)$. Let $E_{bad}$ be the set of cross-pairs with $w_{uv} < a - \eta_0$, meaning $X_{uv} > \eta_0$. By Markov's inequality, $\eta_0 |E_{bad}| \le \sum_{(u,v) \in E_{bad}} X_{uv} \le (\delta_k + O(\varepsilon)) m_k^2$. Because $\delta_k \to 0$ as $k \to \infty$ and $\varepsilon$ can be chosen arbitrarily small, we strictly obtain $|E_{bad}| = o(m_k^2)$. Thus, for all but $o(m_k^2)$ pairs $u \in V_i, v \in V_j$, the original weight satisfies $w_{uv} \ge a - \eta_0$.

Now we bound the number of violating pairs:
\begin{itemize}
    \item Pairs inside $V_0^{(k)}$: at most $|V_0^{(k)}|^2\le \varepsilon^2 m_k^2$.
    \item Pairs with one vertex in $V_0^{(k)}$ and the other outside: at most $2\varepsilon m_k^2$.
    \item Pairs inside the same cluster $V_i$: in $R_k'$ there are no such arcs; in $R_k$ the total number of arcs inside clusters is bounded by the degree loss, at most $(d+\varepsilon)m_k^2/2$.
    \item Pairs coming from clusters that are in the same part $W_p$ but different clusters: in $\widetilde{R}_k$ there are no arcs between them; in $R_k$ the arcs between such clusters are again bounded by the degree loss, at most $(d+\varepsilon)m_k^2$.
\end{itemize}
Thus the total number of arcs violating the ideal structure is at most $(\varepsilon^2+2\varepsilon+(d+\varepsilon)/2+(d+\varepsilon))m_k^2 + o(m_k^2)$, which can be made smaller than $\eta_0 m_k^2$ by choosing $\varepsilon,d$ sufficiently small. This contradicts the assumption that $R_k$ was a counterexample, completing the proof of Lemma~\ref{lem:weighted-stability}.

\subsection{Regularity setup for Theorem~\ref{thm:stability}}
We now start the proof of Theorem~\ref{thm:stability}. Fix $r,t$ and $a$ as in the theorem, and let $\beta>0$ be given. We will choose a chain of constants
\[
\frac1{n_0}\;\ll\;\varepsilon\;\ll\;d\;\ll\;\eta\;\ll\;\beta,\;\frac1r,
\]
where $\eta$ will be the parameter appearing in Lemma~\ref{lem:weighted-stability}. The exact dependencies will become clear during the proof.

Set $s=t(r+1)$ and let $\Delta=\Delta(T_{r+1}^t)$. By the embedding lemma (Lemma~\ref{lem:embedding}) there exists $\varepsilon_0=\varepsilon_0(d,\Delta)$ such that if $\varepsilon\le\varepsilon_0$ and the cluster size $\ell\ge s/\varepsilon_0$, then any blow‑up of a subdigraph of the reduced digraph can be embedded. We choose $\varepsilon\le\min\{\varepsilon_0,d\}$.

Apply the directed regularity lemma (Lemma~\ref{lem:regularity}) to $G$ with parameters $\varepsilon,d$. We obtain a partition $V_0,V_1,\dots,V_k$ of $V(G)$ and a pure digraph $G'\subseteq G$ with the usual properties: $|V_0|\le\varepsilon n$, $|V_1|=\dots=|V_k|=\ell$, $k\ge M'$ (some absolute constant) and every pair $(V_i,V_j)_{G'}$ is either $\varepsilon$-regular with density at least $d$ or has density $0$. Denote $m=k$ and set $R$ to be the reduced digraph on vertex set $\{1,\dots,m\}$ where $ij$ is an arc iff $(V_i,V_j)_{G'}$ is $\varepsilon$-regular with density $\ge d$.

For each such regular pair we define
\[
d_{ij}^2=\frac{\#\{\text{ordered pairs }(u,v)\in V_i\times V_j\text{ with both }uv,vu\in G'\}}{|V_i||V_j|},
\]
\[
d_{ij}^1=\frac{\#\{\text{ordered pairs }(u,v)\in V_i\times V_j\text{ with exactly one of }uv,vu\in G'\}}{|V_i||V_j|},
\]
so that $d_{ij}^1+d_{ij}^2\ge d$. The weighted contribution of this pair to $e_a(G')$ is $(a d_{ij}^2+d_{ij}^1)\ell^2$. We define a weight on the arc $ij$ of $R$ by
\[
w_{ij}=a d_{ij}^2+d_{ij}^1.
\]
Clearly $w_{ij}\in[d,a]$. Moreover, for an ordered pair $(i,j)$ that is not an arc of $R$ we do not assign any weight (it will play no role). The weighted size of $R$ (with these weights) is
\[
e_a^*(R)=\sum_{ij\in E(R)}w_{ij}.
\]

\subsection{From $G$ to $R$}
We now relate $e_a(G)$ to $e_a^*(R)$. From the regularity lemma we have
\[
e_a(G)\le e_a(G')+(d+\varepsilon)n^2.
\]
Inside $G'$ the only possible arcs are between different clusters, and those that belong to non‑regular or low‑density pairs contribute nothing to $e_a(G')$ by definition. Hence
\[
e_a(G')\le a|V_0|n+\sum_{ij\in E(R)}w_{ij}\ell^2\le a\varepsilon n^2+e_a^*(R)\ell^2.
\]
Combining these inequalities yields
\[
e_a(G)\le e_a^*(R)\ell^2+(a\varepsilon+d+\varepsilon)n^2. \tag{1}
\]

On the other hand the hypothesis of Theorem~\ref{thm:stability} gives
\[
e_a(G)\ge a\,t_r(n)-\gamma n^2.
\]
Using the well‑known estimate $t_r(n)=\bigl(1-\frac1r\bigr)\frac{n^2}{2}+O(n)$ and the fact that $\ell\ge(1-\varepsilon)n/m$, we obtain from (1)
\[
e_a^*(R)\frac{(1-\varepsilon)^2n^2}{m^2}\ge a\Bigl(1-\frac1r\Bigr)\frac{n^2}{2}-\gamma n^2-(a\varepsilon+d+\varepsilon)n^2-O(n).
\]
After multiplying by $m^2/n^2$ and setting
\[
\delta:=\gamma+a\varepsilon+d+\varepsilon+o(1),
\]
we get
\[
e_a^*(R)\ge a\Bigl(1-\frac1r\Bigr)\frac{m^2}{2}-\delta m^2. \tag{2}
\]

\subsection{$R$ is $T_{r+1}$-free}
Suppose for a contradiction that $R$ contains a copy of $T_{r+1}$. Then, by the embedding lemma (Lemma~\ref{lem:embedding}) with $s=t(r+1)$, we can embed $T_{r+1}^t$ into $G'$, because each regular pair has density at least $d$ and the cluster size $\ell$ is at least $s/\varepsilon_0$ (since $n$ is large enough). This would contradict the fact that $G$ (and hence $G'$) is $T_{r+1}^t$-free.  Therefore $R$ is $T_{r+1}$-free.

\subsection{Applying the weighted stability lemma to $R$}
Now we apply Lemma~\ref{lem:weighted-stability} to $R$ with the weights $w_{ij}$ defined above. By (2), if we choose $\gamma$ (and consequently $\delta$) sufficiently small, the hypothesis of the lemma is satisfied with $\eta$ (which we have not yet fixed; we will choose it later). Consequently there exists a partition $U_1,\dots,U_r$ of $[m]$ such that, with at most $\eta m^2$ exceptional pairs, we have:
\begin{itemize}
\item no arcs inside the same $U_p$;
\item for $p\neq q$ and $i\in U_p$, $j\in U_q$, both arcs $ij$ and $ji$ belong to $E(R)$ and $w_{ij},w_{ji}\ge a-\eta$.
\end{itemize}

\subsection{Lifting the partition to $G$}
Using this partition of the clusters we now define a partition $X_1,\dots,X_r$ of $V(G)$: put all vertices of $V_i$ into $X_p$ exactly when $i\in U_p$. The exceptional vertices in $V_0$ can be distributed arbitrarily (e.g. equally among the $X_p$). We will show that $G$ differs from $\DT_r(n)$ by at most $\beta n^2$ arcs; here $\DT_r(n)$ means the complete $r$-partite digraph in which every cross pair contains both directions and there are no arcs inside parts. We need to bound the number of arcs that violate this ideal structure.  They can be classified as follows.

\begin{enumerate}
\item \textbf{Arcs incident to $V_0$.}  Since $|V_0|\le\varepsilon n$, there are at most $2\varepsilon n^2$ such arcs.
\item \textbf{Arcs coming from pairs that are not regular or have low density.}  In $G'$ these pairs contribute nothing; in the original $G$ they can have at most $2(d+\varepsilon)n^2$ arcs, because every vertex loses at most $(d+\varepsilon)n$ neighbours in each direction when passing from $G$ to $G'$.
\item \textbf{Arcs corresponding to exceptional cluster pairs.}  By the conclusion of Lemma~\ref{lem:weighted-stability}, there are at most $\eta m^2$ unordered pairs $\{i,j\}$ that violate either the ``no arc inside a part'' condition or the ``full bidirectional arcs with large weight'' condition. Each such pair involves at most $2\ell^2$ arcs (both directions). Hence the total number of arcs in this category is at most $2\eta m^2\ell^2\approx 2\eta n^2$.
\item \textbf{Arcs in ``good'' pairs that are not yet complete bidirectional.}  Consider a pair $(i,j)$ with $i\in U_p$, $j\in U_q$, $p\neq q$, for which both arcs exist and $w_{ij},w_{ji}\ge a-\eta$. What does $w_{ij}\ge a-\eta$ imply about the actual densities $d_{ij}^2,d_{ij}^1$? Since $w_{ij}=a d_{ij}^2+d_{ij}^1$ and $d_{ij}^1= d_{ij}-d_{ij}^2\le 1-d_{ij}^2$, we have
\[
a d_{ij}^2+1-d_{ij}^2\ge a-\eta\quad\Longrightarrow\quad (a-1)d_{ij}^2\ge a-\eta-1.
\]
Because $a>3/2$, $a-1>1/2$, we obtain $d_{ij}^2\ge 1-\frac{\eta}{a-1}$. Thus the density of 2‑cycles in this regular pair is at least $1-\frac{\eta}{a-1}$. To turn this pair into a complete bidirectional pair we may need to add at most $2\frac{\eta}{a-1}\ell^2$ arcs (both directions). The number of such good pairs is at most $\binom{m}{2}\approx\frac{m^2}{2}$, and each contributes at most that many changes. Hence the total number of modifications needed in this class is at most $\frac{r-1}{r}\cdot\frac{\eta}{a-1}n^2$ (the factor $\frac{r-1}{r}$ accounts for the proportion of cross pairs).
\end{enumerate}

Summing these estimates, the total number of arcs that have to be changed is at most
\[
\Bigl(2\varepsilon+2(d+\varepsilon)+2\eta+\frac{r-1}{r}\cdot\frac{\eta}{a-1}\Bigr)n^2+o(n^2).
\]

\subsection{Choice of constants}
Now we choose the constants in the following order:
\begin{itemize}
\item Fix $\beta>0$ as given.
\item Choose $\eta>0$ so small that $2\eta+\frac{r-1}{r}\cdot\frac{\eta}{a-1}<\frac{\beta}{10}$.
\item Pick $d=\frac{\beta}{20}$ and then $\varepsilon\le\min\{\varepsilon_0,\frac{\beta}{20}\}$.
\item Finally choose $\gamma>0$ (in Theorem~\ref{thm:stability}) sufficiently small so that the $\delta$ in (2) is smaller than the $\delta$ required by Lemma~\ref{lem:weighted-stability} for the chosen $\eta$; this is possible because $\delta$ tends to $0$ as $\gamma,\varepsilon,d\to0$.
\end{itemize}
With these choices the total number of changes is less than $\beta n^2$ for all sufficiently large $n$. This completes the proof of Theorem~\ref{thm:stability}.

\section{Typical structure of $T_{r+1}^t$-free oriented graphs and digraphs}

We now combine the container theorem (Theorem~\ref{thm:container}) with the stability result to prove the main theorem.

\begin{lemma}[Rough structure lemma]\label{lem:rough}
Let $r\ge2$, $t\ge1$ be integers and let $a\in(\frac32,2]$. For every $\alpha>0$ there exists $\varepsilon>0$ such that for all sufficiently large $n$,\\
(i)  all but at most $f(n,T_{r+1}^t)2^{-\varepsilon n^{2}}$ $T_{r+1}^t$-free oriented graphs on $n$ vertices can be made $r$-partite by changing at most $\alpha n^{2}$ arcs;\\
(ii)  all but at most $f^{*}(n,T_{r+1}^t)2^{-\varepsilon n^{2}}$ $T_{r+1}^t$-free digraphs on $n$ vertices can be made $r$-partite by changing at most $\alpha n^{2}$ arcs.

Consequently, almost all $T_{r+1}^t$-free oriented graphs and almost all $T_{r+1}^t$-free digraphs are $r$-partite.

Here $f(n,T_{r+1}^t)$ and $f^{*}(n,T_{r+1}^t)$ denote the numbers of labelled $T_{r+1}^t$-free oriented graphs and digraphs on $n$ vertices, respectively.
\end{lemma}

Recall that Lemma \ref{lem:rough} already gives a rough structural description; we now show how to prove it using the container method and stability, and then upgrade it to the exact structure.

\subsection{Proof of the rough structure lemma}
We prove part (i) of Lemma~\ref{lem:rough}; part (ii) is analogous with $a=2$.  Let $a=\log3$. Choose constants $1/n_0\ll\eps\ll\gamma\ll\beta\ll\alpha,1/r$, and set $\eps'=2\eps$.  Apply Theorem~\ref{thm:container} to $H=T_{r+1}^t$ with parameters $N=n$, $\eps'$, obtaining a container family $\cC$. Let $\cC_1=\{G\in\cC:e_a(G)\ge\ex_a(n,T_{r+1}^t)-\eps'n^2\}$.  By the container theorem, $|\cC|\le2^{\eps n^2}$, so the number of $T_{r+1}^t$-free oriented graphs not contained in any $G\in\cC_1$ is at most $|\cC|2^{\ex_a(n,T_{r+1}^t)-\eps'n^2}\le f(n,T_{r+1}^t)2^{-\eps n^2}$.

Now take any $G\in\cC_1$.  By property (b) of containers, $G$ contains at most $\eps'n^{(r+1)t}$ copies of $T_{r+1}^t$. Apply the removal lemma (Lemma~\ref{lem:removal}) to delete at most $\gamma n^2$ arcs and obtain a $T_{r+1}^t$-free digraph $G'$ with $e_a(G')\ge\ex_a(n,T_{r+1}^t)-(\eps'+\gamma)n^2$. By the stability theorem (Theorem~\ref{thm:stability}) with $\beta$, we have $G'=\DT_r(n)\pm\beta n^2$. Hence $G$ itself differs from $\DT_r(n)$ by at most $(\beta+\gamma)n^2\le\alpha n^2$ arcs.  This completes the proof of Lemma~\ref{lem:rough}.

\subsection{Exact structure: almost all graphs have $T_2^t$-free parts}

In this subsection, we upgrade the approximate structural result of Lemma 1.2 to the exact structural characterization. Unlike the $t=1$ case in \cite{kuhn2017structure}, which requires a delicate vertex-by-vertex induction, the $t \ge 2$ case can be resolved elegantly via a relative counting argument inside each container. Let $\mathcal{P}_{n, r, t}$ be the family of oriented graphs on $[n]$ admitting an $r$-partition where each part induces a $T_2^t$-free graph. Any graph in $\mathcal{P}_{n, r, t}$ with all forward cross-edges is strictly $T_{r+1}^t$-free (as any mapping of $T_{r+1}^t$ into such a graph must place at least one $T_2^t$ entirely within some part, which is forbidden). Thus, the number of $T_{r+1}^t$-free graphs is at least $|\mathcal{P}_{n, r, t}| \ge 3^{t_r(n)} 2^{\text{ex}(n/r, T_2^t)}$.

By the container method and the removal lemma (as shown in Section 5.1), all but a negligible fraction of $T_{r+1}^t$-free graphs belong to a small family of containers $\mathcal{C}_1$. Every container $C \in \mathcal{C}_1$ admits an $r$-partition $V_1, \dots, V_r$ such that $C$ is missing at most $\alpha n^2$ forward cross-edges and contains at most $\alpha n^2$ internal edges. Let $\mathcal{F}(C)$ be the family of $T_{r+1}^t$-free oriented graphs contained in $C$. We partition $\mathcal{F}(C)$ into two sets: the ``good'' graphs $\mathcal{G}(C)$, which contain no $T_2^t$ within any part, and the ``bad'' graphs $\mathcal{B}(C)$, which contain at least one $T_2^t$ within some part. We will show that $|\mathcal{B}(C)| \le e^{-c n} |\mathcal{F}(C)|$ for some constant $c > 0$.

For any $G \in \mathcal{F}(C)$, we can uniquely specify $G$ by choosing its internal edges $G_{int} \subset C$ and its cross-edges $G_{cross} \subset C$. Let $N(G_{int})$ be the number of valid choices for $G_{cross}$ such that $G_{int} \cup G_{cross}$ is $T_{r+1}^t$-free. Thus, $|\mathcal{B}(C)| = \sum_{G_{int} \text{ bad}} N(G_{int})$.

Suppose $G_{int}$ is a bad internal configuration. We identify the lexicographically first copy of $T_2^t$ in $G_{int}$, say on a vertex set $X \subset V_i$ with $|X| = 2t$. Because $C$ is missing at most $\alpha n^2$ cross-edges, a standard supersaturation argument \cite{erdos1983supersaturation} shows that the number of canonical copies of $T_{r-1}^t$ in the other $r-1$ parts of $C$ is $\Omega(n^{(r-1)t})$. Since each $T_{r-1}^t$ occupies exactly $t(r-1)$ vertices, we can greedily extract a collection of $M = \delta n$ vertex-disjoint copies of $T_{r-1}^t$, say $Y_1, \dots, Y_M$, for some constant $\delta > 0$.

For $G_{int} \cup G_{cross}$ to be $T_{r+1}^t$-free, $G_{cross}$ must not form a $T_{r+1}^t$ with the $T_2^t$ on $X$. Specifically, for each disjoint copy $Y_j$, the cross-edges between $X$ and $Y_j$ cannot all be oriented perfectly forward. There are $c = 2t^2(r-1)$ pairs of vertices between $X$ and any $Y_j$. Out of the $3^c$ possible edge configurations for these pairs, at least one (the all-forward configuration) is forbidden. Thus, there are at most $3^c - 1$ valid configurations between $X$ and $Y_j$.

Define a mapping $\Phi$ from the bad internal configurations to all internal configurations by deleting all edges within $X$. Let $G'_{int} = \Phi(G_{int})$. Clearly, $G'_{int}$ has no internal edges on $X$. The number of preimages of any $G'_{int}$ is at most $\binom{n}{2t} 3^{\binom{2t}{2}} \le n^{2t}$. Crucially, any valid cross-edge configuration for $G_{int}$ is also valid for $G'_{int}$.

Since the copies $Y_1, \dots, Y_M$ are vertex-disjoint, the choices of cross-edges between $X$ and each $Y_j$ are independent. Let $m$ be the maximum total number of cross-edges from $X$ to the other parts. The number of valid choices for $G_{cross}$ incident to $X$ in $G_{int}$ is therefore restricted by the forbidden subgraphs and is bounded by:
$$(3^c - 1)^{\delta n} \cdot 3^{m - c \delta n} = 3^m \left( 1 - 3^{-c} \right)^{\delta n}.$$
By setting $\gamma = -\delta \ln(1 - 3^{-c}) > 0$, this number is strictly bounded by $3^m e^{-\gamma n}$. In contrast, $G'_{int}$ places no such restriction on $X$, allowing all $3^m$ orientations up to the capacity of $C$. This rigorously establishes the strict relative bound $N(G_{int}) \le e^{-\gamma n} N(G'_{int})$.

Summing over all bad configurations gives:
$$|\mathcal{B}(C)| = \sum_{G_{int} \text{ bad}} N(G_{int}) \le \sum_{G'_{int}} n^{2t} e^{-\gamma n} N(G'_{int}) \le n^{2t} e^{-\gamma n} |\mathcal{F}(C)|.$$
For sufficiently large $n$, we have $n^{2t} e^{-\gamma n} = o(1)$, meaning almost no graphs in $C$ contain a $T_2^t$ inside any part. Summing over the small family of containers $C \in \mathcal{C}_1$, we conclude that almost all $T_{r+1}^t$-free graphs belong to $\mathcal{G}(C)$ for some $C$, which precisely means they admit an $r$-partition where each part is $T_2^t$-free. This completely establishes the exact typical structure, finishing the proof of Theorem 1.1 and Theorem 1.2. The identical argument applies to digraphs by substituting $3$ with $4$ for the cross-edge orientation choices.

\section{Concluding remarks and open problems}

We have shown that for any $r\ge2$, $t\ge1$, and any weight parameter $a\in(3/2,2]$, almost all $T_{r+1}^t$-free oriented graphs and digraphs are $r$-partite. This generalises the KOTZ theorem and confirms a conjecture of Liang and Liu \cite{liang2022typical}. Our proof relies on a weighted extremal theorem, a stability result, and the container method. The restriction $a>3/2$ is technical and comes from certain inequalities in the stability proof; it would be interesting to extend the result to all $a\ge1$. 

Another natural direction is to consider blow-ups of other tournaments or cycles. The methods developed here should apply as long as the forbidden digraph has a suitable extremal structure (e.g., a complete $r$-partite digraph is extremal). For cycles, the situation is more complex and already studied in \cite{kuhn2017structure}.

\end{document}